\documentclass[11pt, reqno, twoside]{amsart}
\usepackage[english]{babel}
\usepackage{graphicx,color}
\usepackage[all]{xy}
\usepackage{amsfonts,dsfont,mathrsfs}
\usepackage{amsmath, amsthm, amssymb}
\usepackage{enumerate, url}
\usepackage{fancyhdr}

\newcommand{\bbC}{\ensuremath{\mathbb{C}}}
\newcommand{\bbR}{\ensuremath{\mathbb{R}}}
\newcommand{\bbQ}{\ensuremath{\mathbb{Q}}}
\newcommand{\bbF}{\ensuremath{\mathbb{F}}}
\newcommand{\bbZ}{\ensuremath{\mathbb{Z}}}

\newcommand{\frh}{\ensuremath{\mathfrak{h}}}

\makeatletter
\DeclareRobustCommand{\tvdots}{
  \vbox{\baselineskip4\p@\lineskiplimit\z@\kern0\p@\hbox{.}\hbox{.}\hbox{.}}}
\makeatother

\DeclareMathOperator{\Aut}{Aut}

\numberwithin{equation}{section}

\newtheorem{thmnr}{Theorem}[section]

\newtheorem{propnr}[thmnr]{Proposition}

\newtheorem{lemnr}[thmnr]{Lemma}

\newtheorem{cornr}[thmnr]{Corollary}
\theoremstyle{definition}

\newtheorem{dfnnr}[thmnr]{Definition}

\newtheorem{rmknr}[thmnr]{Remark}

\newtheorem{exsnr}[thmnr]{Examples}

\newtheorem{claimnr}[thmnr]{Claim}

\newtheorem{probnr}[thmnr]{Problem}

\newtheorem{quesnr}[thmnr]{Question}

\begin{document}

\title{Towers of regular self-covers and linear endomorphisms of tori}

\author[W. van Limbeek]{Wouter van Limbeek}
\address{Department of Mathematics \\ 
                 University of Michigan \\
                 Ann Arbor, MI 48109}

\date{\today}

\begin{abstract} Let $M$ be a closed manifold that admits a self-cover $p:M\rightarrow M$ of degree $>1$. We say $p$ is strongly regular if all iterates $p^n:M\rightarrow M$ are regular covers. In this case, we establish an algebraic structure theorem for the fundamental group of $M$: We prove that $\pi_1(M)$ surjects onto a nontrivial free abelian group $A$, and the self-cover is induced by a linear endomorphism of $A$. Under further hypotheses we show that a finite cover of $M$ admits the structure of a principal torus bundle. We show that this applies when $M$ is K\"ahler and $p$ is a strongly regular, holomorphic self-cover, and prove that a finite cover of $M$ splits as a product with a torus factor.
\end{abstract}

\maketitle

\tableofcontents

\section{Introduction}
\label{sec:intr}

\subsection{Main results} 
Let $M$ be a closed manifold that admits a smooth covering map $p:M\rightarrow M$. We say $p$ is \emph{nontrivial} if deg$(p)>1$.

\begin{quesnr} Which closed manifolds $M$ admit a nontrivial self-covering?\end{quesnr}

Key examples are furnished by tori: If $M=T^n$ and $A$ is an integer $n\times n$-matrix, then the linear map $A:\bbR^n\rightarrow \bbR^n$ descends to a map $\overline{A}:T^n\rightarrow T^n$ with $\deg(\overline{A})=\det(A)$. Therefore a torus admits many nontrivial self-coverings.

More generally suppose $N$ is a simply-connected nilpotent Lie group and $\Gamma\subseteq N$ is a lattice, and $\alpha:N\rightarrow N$ is an automorphism with $\alpha(\Gamma)\subseteq \Gamma$. Then $\alpha$ descends to a map $N\slash\Gamma\rightarrow N\slash\alpha(\Gamma)$. Postcomposition with the natural projection $N\slash\alpha(\Gamma)\rightarrow N\slash\Gamma$ yields a self-cover of $N\slash\Gamma$ with degree $[\Gamma:\alpha(\Gamma)]$.

For some choices of $N$ and $\Gamma$, there exists $\alpha$ such that this construction yields a nontrivial self-cover (e.g. choosing $N$ (resp. $\Gamma$) to be the real (resp. integer) Heisenberg group). Belegradek has made a detailed investigation of these self-covers \cite{igornilp}, and in particular proved that some nilmanifolds do not nontrivially self-cover.

A modification of the above examples by finite covers gives nontrivial self-covers of flat (resp. infranil) manifolds, e.g. of the Klein bottle. Finally we can perform fiber bundle constructions with fibers of the types discussed above, and where the self-cover occurs along the fiber. For example, if $T$ is a torus, and $M\rightarrow N$ is a flat principal $T$-bundle such that the monodromy $\pi_1(N)\rightarrow \Aut(T)$ commutes with a self-cover $A:T\rightarrow T$, then $A$ induces a self-cover of $M$.

Are these all the examples (up to finite covers)? Progress on this question has been very limited. The only results are either in dimensions $\leq 3$ (by Yu-Wang \cite{3dim}, also see \cite{wangicm}), or for so-called \emph{expanding maps} (see Section \ref{sec:res}) and to some extent for algebraic maps on smooth projective varieties (see Section \ref{sec:kaehlerresult}). All of these results suggest that there are no other examples.

In this paper we start the classification of self-covers without placing geometric constraints on the map or the space. To this end we make the following definition:
\begin{dfnnr}
Let $X$ be a topological space. A self-cover $p:X\rightarrow X$ is called \emph{strongly regular} if the covers $p^n:X\rightarrow X$ are regular for every $n\geq 1$. 
\end{dfnnr} 
\begin{rmknr} It is easy to see that a characteristic cover (i.e. $p_\ast(\pi_1(M))\subseteq \pi_1(M)$ is characteristic) is strongly regular.\end{rmknr}

Our first result shows that at least on the level of fundamental groups, any strongly regular self-cover comes from a linear endomorphism on a torus. In fact this theorem does not just hold for smooth manifolds, but in the more general setting of finite CW complexes.
\begin{thmnr} Let $X$ be a finite CW complex, and let $p:X\rightarrow X$ be a continuous strongly regular self-cover with degree $1<d<\infty$. Set $\Gamma:=\pi_1(X)$ and let $\varphi:\Gamma\hookrightarrow\Gamma$ be the map induced by $p$. 

Then there exists a surjective morphism $q:\Gamma\rightarrow \bbZ^n$ for some $n\geq 1$, and such that $\varphi$ restricts to an automorphism of $\ker(q)$ and descends to a linear endomorphism of $\bbZ^n$ with determinant $d$.
	\label{thm:group}
	\end{thmnr}
We have the following immediate corollary of Theorem \ref{thm:group}.
	\begin{cornr} Let $X$ be a finite CW-complex that admits a nontrivial strongly regular self-cover. Then $b_1(X)>0$. \label{cor:betti} \end{cornr}
	Recall that for a finite CW-complex $X$ with fundamental group $\Gamma$, any homomorphism $\Gamma\rightarrow \bbZ^n$ is induced by some continuous map $X\rightarrow T^n$, and two maps $X\rightarrow T^n$ are homotopic if and only if they induce the same map on fundamental groups. Therefore Theorem \ref{thm:group} has the following purely topological phrasing:
	\begin{cornr} Let $X$ be a finite CW-complex and $p:X\rightarrow X$ a nontrivial strongly regular self-cover of degree $1<d<\infty$. Then there exist
		\begin{itemize}
		\item a map $h:X\rightarrow T^n$ for some $n\geq 1$, 
		\item and a linear endomorphism $A:T^n\rightarrow T^n$ with determinant $d$
	\end{itemize}
such that the square
	$$\xymatrix{
			&X \ar[d]^{p} \ar[r]^{h}& T^n \ar[d]^A\\
			&X \ar[r]^h & T^n
			}$$
	commutes up to homotopy.
	\label{cor:htpy}
	\end{cornr} 
Note that in Corollary \ref{cor:htpy}, $h$ is only determined up to homotopy. It is then natural to wonder if $h$ can be chosen to be of a particularly nice topological form. We do not know, but let us at least ask:
\begin{quesnr} Let $X$ be a closed smooth manifold and $p:X\rightarrow X$ a nontrivial strongly regular smooth self-cover. Can we choose $h$ to be a topological fiber bundle?\end{quesnr}
\begin{rmknr} Note that it is definitely not true that $h$ can be chosen to be a smooth fiber bundle: Namely, Farrell and Jones showed that there are tori with exotic smooth structures that admit expanding maps \cite{exoticexp}. Such a map is a (smooth) nontrivial strongly regular self-cover, and $h:X\rightarrow T$ is homotopic to the identity. In this case one can choose $h$ to be a homeomorphism but not a diffeomorphism.

\label{rmk:smooth}
\end{rmknr}
In the proof of Theorem \ref{thm:group}, it will be essential to study the actions of the deck groups $F_n$ of the covers $p^n:X\rightarrow X$. In fact, we will see that as subgroups of Homeo$(X)$, we have $F_n\subseteq F_{n+1}$ for every $n$. Hence we can obtain a locally finite group $F:=\cup_n F_n$ that acts on $X$. We call $F$ the \emph{asymptotic deck group} of $p:M\rightarrow M$ (see Section \ref{sec:group} for more information). Analysis of this group and its action will be the driving force behind the proof of Theorem \ref{thm:group}. We will now restrict to the case where $X=M$ is a closed manifold, so that we can introduce a special class of self-covers for which the action of $F$ on $M$ is well-behaved.
\begin{dfnnr} Let $p:M\rightarrow M$ be a nontrivial, strongly regular self-cover. We say $p$ is \emph{proper} if $F\subseteq \text{Homeo}(M)$ is precompact. If in addition $p$ is smooth, we say $p$ is \emph{smoothly proper} if $F$ is precompact in Diff$(M)$ with respect to the $C^\infty$-topology. \end{dfnnr}
For smoothly proper self-covers we have the following structure theorem.
\begin{thmnr} Let $M$ be a closed manifold and suppose $p:M\rightarrow M$ is a nontrivial, strongly regular, smoothly proper self-cover. Then there exist a finite cover $M'\rightarrow M$, and a compact torus $T$ such that
	\begin{enumerate}[(i)]
	\item $M'$ is a smooth principal $T$-bundle $\pi: M'\rightarrow N$ over some closed manifold $N$,
	\item $\pi_1(M') \cong \pi_1(T)\times \pi_1(N)$,
	\item $p$ lifts to a strongly regular self-cover $p':M'\rightarrow M'$,
	\item $p'$ is a bundle map (with respect to $\pi$) that descends to a diffeomorphism of $N$, and
	\item on each $T$-orbit in $M'$, $p'$ restricts to a covering of degree $\deg(p)$.
	\end{enumerate}
	\label{thm:proper}
\end{thmnr}
It seems reasonable to expect that if $p$ preserves some additional structure of a geometric nature, then $p$ will be smoothly proper. Below we will establish this in the case that $M$ is a K\"ahler manifold and $p$ is holomorphic (see Theorem \ref{thm:kaehler}).

\subsection{Holomorphic endomorphisms of K\"ahler manifolds} 
\label{sec:kaehlerresult}
Recall that if $M$ is a closed complex manifold, then a \emph{K\"ahler structure} on $M$ is a Hermitian metric whose associated 2-form is closed. 

For example, any smooth complex projective variety is a K\"ahler manifold.  If $X$ is a smooth complex projective variety, the natural class of self-covers $f:X\rightarrow X$ to consider are the \emph{\'etale} maps. \'Etale maps have been intensely studied, but a complete classification (up to finite \'etale cover) is only known in low dimensions \cite{projlow,projlow2}. In high dimensions the picture is more fragmented, see e.g. \cite{proj1, proj2} for more information. For nonalgebraic K\"ahler manifolds, a complete classification has been obtained in dimensions up to 3 by H\"oring-Peternell \cite{kaehlerlow}.  

We prove that strongly regular holomorphic self-covers on K\"ahler manifolds are smoothly proper. Combined with Theorem \ref{thm:proper}, this allows us to obtain a splitting with a torus factor of a finite cover of $M$:

\begin{thmnr} Let $M$ be a closed K\"ahler manifold and $p:M\rightarrow M$ a nontrivial, strongly regular, holomorphic self-cover. Then there exist a finite cover $M'\rightarrow M$ of $M$, a complex torus $T$ and a K\"ahler manifold $N$ such that
	\begin{enumerate}[(i)]
		\item $M'$ is biholomorphic to $T\times N$,
		\item $p$ lifts to a smoothly proper, strongly regular self-cover $p':M'\rightarrow M'$,
		\item $p'$ descends to a holomorphic automorphism of $N\cong M'\slash T$.
		\item on each slice $T\times \{x\}, x\in N$, the map $p'$ restricts to a holomorphic cover of degree $\deg(p)$.
	\end{enumerate}
	\label{thm:kaehler}
\end{thmnr}

\subsection{Residual covers} \label{sec:res} Finally we will discuss applications of Theorem \ref{thm:group} to so-called expanding maps and strongly scale-invariant groups. First let us make the following definition.
	\begin{dfnnr} Let $p:M\rightarrow M$ be a self-cover. Consider the associated tower of covers
		$$\dots\overset{p}{\longrightarrow} M\overset{p}{\longrightarrow} M\overset{p}{\longrightarrow} M$$
	We say $p$ is \emph{residual} if this tower is residual, i.e. $\displaystyle\bigcap_{n\geq 1} p_\ast^n(\pi_1(M))=1$.
	\label{dfn:res}
	\end{dfnnr}
We give the following application of Theorem \ref{thm:group} to strongly regular residual self-covers:
	\begin{cornr} Let $X$ be a finite CW-complex admitting a strongly regular, residual self-cover $p$. Then $\pi_1(X)$ is abelian. In particular, if $X$ is in addition a closed aspherical manifold, then $X$ is homeomorphic to a torus.\label{cor:charres}\end{cornr}
A group-theoretic version of residual self-covers has been studied by Benjamini and Nekrashevych-Pete \cite{scaleinv}. Let $\Gamma$ be a finitely generated group. Then we define:
	\begin{itemize}
		\item (Benjamini) $\Gamma$ is called \emph{scale-invariant} if $\Gamma$ admits a decreasing residual chain
			$$\Gamma=\Gamma_0 \supseteq \Gamma_1 \supseteq \Gamma_2 \supseteq \dots$$
			where $\Gamma_n$ is of finite index in $\Gamma$ and $\Gamma_n\cong\Gamma$ for every $n\geq 0$.
			\item (Nekrashevych-Pete) $\Gamma$ is called \emph{strongly scale-invariant} if there exists an embedding $\varphi:\Gamma\hookrightarrow \Gamma$ with image of finite index, such that $\cap_n \varphi^n(\Gamma)=1$.
	\end{itemize}
Scale-invariant groups were introduced by Benjamini in a blog post, in which he asked if any such group is virtually nilpotent. Counterexamples were given by Nekrashevych-Pete \cite{scaleinv}, after which they defined the above notion of a strongly scale-invariant group, and asked whether a strongly scale-invariant group is virtually nilpotent. 

Hence Corollary \ref{cor:charres} can be interpreted as giving a positive answer to a topological version of the question of Nekrashevych-Pete, under the additional assumption that $\varphi^n(\Gamma)$ is normal in $\Gamma$ for every $n\geq 1$.

As a special case of residual self-covers we consider expanding maps:
\begin{dfnnr} Let $M$ be a closed manifold and choose an auxiliary Riemannian metric on $M$. Then a $C^1$ map $f:M\rightarrow M$ is \emph{expanding} if there exist $c>0$ and $\lambda>1$ such that for every $m\geq 1$ and $v\in TM$, we have
	$$||Df^m(v)||\geq c \lambda^m ||v||.$$
\end{dfnnr}
\begin{rmknr} Whether or not a map $f:M\rightarrow M$ is expanding does not depend on the choice of Riemannian metric on $M$.\end{rmknr}
Any expanding map $f:M\rightarrow M$ is a local homeomorphism, and hence a covering map. Expanding maps on closed Riemannian manifolds were classified by the combined work of Franks \cite{franksexp}, Shub \cite{shubexp}, and Gromov \cite{polygrowth}: Any expanding map on a closed manifold is topologically conjugate to an algebraic endomorphism of an infranilmanifold. 

By an approach that is different from the Franks-Shub-Gromov proof (in particular without use of Gromov's polynomial growth theorem), we show that strongly regular expanding maps are exactly given by linear endomorphisms of tori (up to topological conjugacy and finite covers).
	\begin{cornr}[Characterization of strongly regular expanding maps] Let $M$ be a closed manifold and $p:M\rightarrow M$ a $C^1$ strongly regular expanding map. Then $M$ is diffeomorphic to a torus and $p$ is topologically conjugate to a linear expanding map.\label{cor:charexp}\end{cornr}
	
		\subsection{Outline of the paper} In Section \ref{sec:locfin} we review some preliminaries regarding locally finite groups, especially the theory of Artinian locally finite groups. In Section \ref{sec:group} we prove Theorem \ref{thm:group} that establishes the connection between strongly regular self-covers factor and linear endomorphisms of tori. 
		
		We start by introducing the key object for the proof, the asymptotic deck group $F$ associated to a strongly regular self-cover. The key result is that $F$ is abelian and in fact a finite product of finite cyclic groups and quasicyclic groups. This allows us to construct the desired map $\pi_1(X)\rightarrow \bbZ^n$.
		
		In Section \ref{sec:proper} we prove Theorem \ref{sec:proper} that shows that a smoothly proper self-cover is virtually given by a principal torus bundle. By using the Hilbert-Smith conjecture for smooth maps, we know that the closure of the asymptotic deck group is a compact Lie group whose components are tori. Most of the proof is devoted to showing that there is a finite cover where the torus action is free. The quotient by this free action is the desired principal bundle.
		
		In Section \ref{sec:kaehler} we combine the above characterization of smoothly proper self-covers with K\"ahler geometry to show that a K\"ahler manifold with a strongly regular, holomorphic self-cover has a finite cover that splits with a torus factor. Finally in Section \ref{sec:respf} we prove Corollaries \ref{cor:charres} and \ref{cor:charexp} about residual self-covers and expanding maps.
		
		\subsection{Acknowledgments:} I would like to thank Jim Davis, Benson Farb, Ralf Spatzier, Bena Tshishiku, Zhiren Wang and Shmuel Weinberger for helpful conversations. I am grateful to David Fisher for alerting me to the notion of scale-invariant groups. Part of this work was completed at the University of Chicago. I am grateful for its financial support.
	
\section{Preliminaries on locally finite groups}
\label{sec:locfin}
A key tool in the proofs of our results on self-covers is the emergence of a locally finite group. In this section we review the definition and facts needed for the rest of the paper. For more information we refer to \cite{locfin}.

\begin{dfnnr} A group $G$ is \emph{locally finite} if any finitely generated subgroup is finite.\end{dfnnr}
\begin{rmknr} As any group is the direct limit of its finitely generated subgroups, a locally finite group is the direct limit of a system of finite groups. Conversely, a direct limit of a system of finite groups is easily seen to be locally finite.\end{rmknr}
\begin{exsnr}\mbox{}
\begin{enumerate}[(i)]
	\item Any finite group is locally finite.
	\item Any direct sum of finite groups is locally finite.
	\item $\bbQ\slash \bbZ$ is locally finite. This is the group of complex roots of unity.
	\item Let $p\geq 2$ be a prime. The \emph{$p$-quasicyclic group} (or  \emph{Pr\"ufer $p$-group}) is the group $\bbZ(p^\infty):=\bbZ[\frac{1}{p}]\slash \bbZ$. This is the group of complex roots of unity whose order is a power of $p$.
\end{enumerate}
\end{exsnr}
An especially well-studied class of locally finite groups is formed by those that are Artinian. Recall that a group $G$ is called \emph{Artinian} if any nonincreasing chain 
$$G=G_0\supseteq G_1 \supseteq G_2\supseteq \dots$$
of subgroups of $G$ eventually stabilizes. This is also called the \emph{minimality condition}.

Among the above examples, the finite groups and quasicyclic groups are Artinian, whereas $\bbQ\slash \bbZ$ and an infinite direct sum of nontrivial finite groups are not Artinian. In 1940, Chernikov posed:

\begin{probnr}[Chernikov's Minimality Problem] Is any Artinian group a finite extension of an abelian Artinian group? \label{prob:min} \end{probnr}

In general this is false, as shown by Ol'shanskii's construction of Tarski monsters \cite{tarski}. In view of the above Minimality Problem, it seems useful to characterize abelian Artinian groups. This was accomplished by Kurosh:
	\begin{thmnr}[{Kurosh, see e.g. \cite[4.2.11]{robingps}}] 
	Let $G$ be an abelian group. Then $G$ is Artinian if and only if $G$ is a finite direct sum of quasicyclic and finite cyclic groups. \label{thm:chernartin}
	\end{thmnr}
For general locally finite groups, Shunkov gave the following criterion for being Artinian:
	\begin{thmnr}[{Shunkov \cite{abartin}}] Let $G$ be a locally finite group. Then $G$ is Artinian if and only if any abelian subgroup of $G$ is Artinian. 
	\label{thm:abartin}
	\end{thmnr}
Shunkov then used this to solve Problem \ref{prob:min} for general locally finite groups. An independent solution (without use of Shunkov's criterion \ref{thm:abartin}) was obtained by Kegel-Wehrfritz \cite{kwartin}. 
	\begin{thmnr}[Shunkov, Kegel-Wehrfritz] A locally finite group $G$ is Artinian if and only if $G$ is a finite extension of an abelian Artinian group. \label{thm:artin} \end{thmnr}

In view of Shunkov's criterion (Theorem \ref{thm:abartin}) it seems useful to have an easy method for determining whether a locally finite abelian group is Artinian. To this end, we recall the notion of rank:
	\begin{dfnnr} Let $G$ be a periodic abelian group. For any prime $p$, consider the set of elements $G(p)$ of order $p$. Then $G(p)$ is an elementary abelian $p$-group, and hence is a vector space over $\bbF_p$. The \emph{$p$-rank} of $G$ is $r_p(G):=\dim_{\bbF_p} G(p)$. The \emph{rank} of $G$ is $r(G):=\sum_p r_p(G)$. \label{dfn:rk} \end{dfnnr}
	Since finite and quasicyclic groups have bounded rank, we see (by use of Kurosh's theorem) that any periodic, Artinian, abelian group has bounded rank. In fact, the converse is also true. To see this, one first uses the primary decomposition theorem for abelian groups (see e.g. \cite[4.1.1]{robingps}) to reduce the problem to $p$-groups for some prime $p$, and for $p$-groups the statement is exactly \cite[4.3.13]{robingps}. We summarize the result as follows:
	\begin{propnr} Let $A$ be a periodic abelian group. Then $A$ is Artinian if and only if $A$ has bounded rank. \label{prop:artincrit}\end{propnr}
	
\section{The asymptotic deck group}
\label{sec:group}
The goal of this section is to prove the following.
	\begin{thmnr} Let $X$ be a finite CW-complex, and let $p:X\rightarrow X$ be a continuous strongly regular self-cover with degree $1<d<\infty$. Set $\Gamma:=\pi_1(X)$ and let $\varphi:\Gamma\hookrightarrow\Gamma$ be the map induced by $p$. 
	
	Then there exist $n\geq 1$ and a surjective morphism $q:\Gamma\rightarrow \bbZ^n$ such that $\varphi$ restricts to an automorphism of $\ker(q)$ and descends to a linear endomorphism of $\bbZ^n$ with determinant $d$.
	\label{thm:group2}
	\end{thmnr}
For the rest of this section we retain the notation of Theorem \ref{thm:group2}. We start by introducing a key object in the proof.
	\begin{dfnnr} Consider the inductive system
	$$1\rightarrow \Gamma\slash \varphi(\Gamma) \overset{\varphi}{\longrightarrow} \Gamma\slash \varphi^2(\Gamma) \overset{\varphi}{\longrightarrow} \Gamma\slash\varphi^3(\Gamma) \overset{\varphi}{\longrightarrow}\dots$$
and write $F_n:=\Gamma\slash\varphi^n(\Gamma)$. The limit $F:=\varinjlim F_n$ of the above system is the \emph{asymptotic deck group} of $p$.\end{dfnnr}
	The asympotic deck group of $p$ is therefore a locally finite group that is the limit of the deck groups $F_n$ of $p^n$ as $n\rightarrow \infty$. Since $\varphi$ is injective, we can write $F=\cup_n F_n$, where we identify $F_n$ with a subgroup of $F_{n+1}$ via
			\begin{align}
			F_n	&=\Gamma\slash\varphi^n(\Gamma) \nonumber\\
				&\cong \varphi(\Gamma)\slash \varphi^{n+1}(\Gamma) \nonumber\\
				&\subseteq \Gamma\slash \varphi^{n+1}(\Gamma)=F_{n+1}.
				\label{eq:identify}
			\end{align}
	Each $F_n$ acts on $X$ as the group of deck transformations of $p^n$. A key observation is that these actions are compatible with the identification in Equation \eqref{eq:identify}, so that we have:
	\begin{lemnr} There is an action of $F$ on $X$ that extends, for every $n\geq 1$, the action of $F_n$ on $X$ by deck transformations of $p^n$.
	\label{lem:extact}
	\end{lemnr} 
	\begin{proof} Let $\widetilde{X}$ be the universal cover of $X$, so that we can write $X=\widetilde{X}\slash\Gamma$. For any $n\geq 1$, we have a natural map $X\slash \varphi^n(\Gamma)\rightarrow X\slash\Gamma$ given by taking the quotient by the action of $F_n=\Gamma\slash\varphi^n(\Gamma)$. The  map $p^n:X\rightarrow X$ can be lifted to a homeomorphism
	$$p_n:\widetilde{X}\slash\Gamma\rightarrow \widetilde{X}\slash \varphi^n(\Gamma).$$
Choose such lifts $(p_n)_{n\geq 1}$ that are compatible, so that $p_{n+1}$ is a lift of $p_n\circ p$, i.e. such that for every $n\geq 1$ the triangle
	\begin{equation}\begin{gathered} \xymatrix{
											& \widetilde{X}\slash\varphi^{n+1}(\Gamma) \ar[dr]^{\hspace{0.5 cm}\slash(F_n\slash F_{n+1})}\\
	X \ar[ur]^{p_{n+1}} \ar[r]_{\hspace{-0.5cm} p} 		& X \ar[r]_{\hspace{-0.5 cm} p_n} 							& \widetilde{X}\slash\varphi^n(\Gamma) 
	}\end{gathered}\label{eq:lift}\end{equation}
commutes. Also choose lifts
	$$q_n:\widetilde{X}\slash\varphi^n(\Gamma)\rightarrow \widetilde{X}\slash\varphi^{n+1}(\Gamma)$$
	of $p$ such that $p_{n+1}=q_n\circ p_n$. 
	
	Let $\varphi_n:F_n\rightarrow F_{n+1}$ be the map of Equation \eqref{eq:identify}. Note that $q_n$ is $\varphi_n$-equivariant (because $p$ induces $\varphi$ on fundamental groups and $q_n$ is a lift of $p$). 
	
	Now define an action of $F_n$ on $X$ by conjugating the action of $F_n$ on $\widetilde{X}\slash\varphi^n(\Gamma)$ by $p_n^{-1}$, i.e. for $g\in F_n$ and $x\in X$ we set
		$$g\cdot x = p_n^{-1} g p_n(x).$$
	It remains to show that for $g\in F_n$, the action of $g$ on $\varphi_n(g)\in F_{n+1}$ coincide. This is immediate, using that $p_{n+1}=q_n\circ p_n$ and that $q_n$ is $\varphi_n$-equivariant.\end{proof}
	The natural quotient maps
		$$\Gamma\slash \varphi^{n+1}(\Gamma)\rightarrow \Gamma\slash \varphi^n(\Gamma)$$
	induce a surjective map $\alpha: F\rightarrow F$. We will not need the following result until Section \ref{sec:proper}, but it will be convenient to prove it here:
	\begin{lemnr} $p$ is $\alpha$-equivariant. \label{lem:aeq}\end{lemnr}
\begin{proof} We fill in the middle map in the diagram of Equation \eqref{eq:lift} to obtain
\begin{equation}\begin{gathered} \xymatrix{
											& \widetilde{X}\slash\varphi^{n+1}(\Gamma) \ar[d] \ar[dr]^{\hspace{0.5 cm}\slash(F_n\slash F_{n+1})}\\
	X \ar[ur]^{p_{n+1}} \ar[r]_{\hspace{-0.5cm} p} 		& X \ar[r]_{\hspace{-0.5 cm} p_n} 							& \widetilde{X}\slash\varphi^n(\Gamma) 
	}\end{gathered}\label{eq:liftadd}\end{equation}
Using the left triangle, we see the vertical map is given by $p\circ p_{n+1}^{-1}$. On the other hand by using the right triangle, the map is given by $p_n^{-1}\circ \alpha_{n+1}$, where $\alpha_{n+1}$ is the natural quotient map. Thus we obtain that
	$$p\circ p_{n+1}^{-1} =p_n^{-1}\circ \alpha_{n+1}.$$
Now the proof of equivariance is straightforward from the definitions and commutativity of Diagram \ref{eq:liftadd} : For $g\in F_{n+1}$, we have
	\begin{align*}
	p(gx)	&=p(\,p_{n+1}^{-1} \,g \,p_{n+1}(x))\\
		&=p_n^{-1} \alpha_{n+1} \,g \,p_{n+1}(x)\\
		& = p_n^{-1} \alpha(g) \, \alpha_{n+1} \, p_{n+1}(x)\\
		& = p_n^{-1} \alpha(g) \, p_n \, p(x)\\
		&=\alpha(g) p(x).
	\end{align*} 
\end{proof}
	The action given by Lemma \ref{lem:extact} gives us strong structural constraints on $F$:				
	\begin{lemnr} $F$ is Artinian. \label{lem:artin}\end{lemnr}
	\begin{proof} By Shunkov's criterion (Theorem \ref{thm:abartin}), it suffices to prove any abelian subgroup of $F$ is Artinian. Suppose $A\subseteq F$ is an abelian subgroup. Since $A$ is countable and abelian, we can choose an increasing chain
		$$A_1\subsetneq A_2 \subsetneq A_3 \subsetneq \dots$$
	such that $A_n$ is finitely generated for all $n$ and $\cup_n A_n = A$. Since $A_n$ is a finitely generated subgroup of the locally finite group $F$, it follows that $A_n$ is finite. Therefore there exist $K_n\geq 1$, and primes $p_{n,k}$ and integers $l_{n,k}\geq 1$ (with $1\leq k\leq K_n$) such that
		\begin{equation} A_n \cong \bigoplus_{1\leq k\leq K_n} \bbZ \slash p_{n,k}^{l_{n,k}} \bbZ. \label{eq:decomp}\end{equation}
In fact only finitely many primes are allowed to appear as the primes $p_{n,k}$ in the decomposition \eqref{eq:decomp} of $A_n$ : Namely, since $A_n$ is finite and $F=\cup_N F_N$, we see that there exists $N$ such that $A_n \subseteq F_N$. Note that $F_N$ is the deck group of the covering map $p^N:M\rightarrow M$, so that the order of $F_N$ is $d^N=\deg(p)^N$. Since $A_n$ is a subgroup of $F_N$, its order $|A_n|$ also divides $d^N$. Therefore $p_{n,k}$ divides $d$ for every $n$ and $k$. 

Next we will bound the number of factors $K_n$ of the decomposition of $A_n$ given by Equation \eqref{eq:decomp}. For this, we will need the following result about finite transformation groups, which has been proven for $\ell=2$ by Carlsson \cite{rk2act} and for $\ell>2$ by Baumgartner \cite{rkpact}. For simplicity, we will only state a special case of their results for homologically trivial actions.
		\begin{thmnr}[Baumgartner, Carlsson] Let $X$ be a finite CW-complex and $\ell$ prime. Let $G\cong (\bbZ\slash \ell\bbZ)^r$ be an elementary abelian $\ell$-group acting effectively and freely on $X$ by cellular maps, and suppose that $G$ acts trivially on $H_\ast(X;\bbF_\ell)$. 
		
		Then $H_j(X;\bbF_\ell)\neq 0$ for at least $r+1$ values of $j$.
		\label{thm:rankact}\end{thmnr}
	\begin{claimnr} Let $K_n$ be as in Equation \eqref{eq:decomp}. Then there is a uniform bound on $K_n$. \label{cl:rkbd} \end{claimnr}
	\begin{proof} It suffices to prove that for every prime divisor $\ell$ of $d=\deg(p)$, the $\ell$-rank of $A_n$ is uniformly bounded. (For the notion of rank, see Definition \ref{dfn:rk}.) Let $\ell$ be a prime divisor of $d$ and fix $n\geq 1$. Suppose that $r\geq 0$ is such that $G:=(\bbZ\slash \ell\bbZ)^r\subseteq A_n$. 
	Choose $N\geq 1$ such that $A_n\subseteq F_N$. Since $F_N$ acts by deck transformations on $X$, we see that $G$ acts freely on $X$. 
	
	Next we show that we can choose a CW-structure on $X$ so that $G$ acts by cellular maps. Indeed, fix any finite CW-structure on $X$ and lift it through the cover $p^N:X\rightarrow X$. This induces a cell structure on the cover so that deck transformations act by cellular maps. Whitehead proved that the topology induced by this lifted cell structure coincides with the topology of the cover \cite[(N)]{cellcover}. Note that the lifted CW-structure is necessarily different from the original one (because it has more cells), but the dimension of $X$ with respect to the two CW-structures is the same.
	
	We want to apply Theorem \ref{thm:rankact} to the action of $G$ on $X$. We already know that $G$ acts freely, and $X$ admits a cell-structure so that $G$ acts cellularly, but of course $G$ may not act trivially on homology. The action of $G$ on homology of $X$ with $\bbF_\ell$-coefficients is a representation of $G$ over $\bbF_\ell$:
		$$\rho:G\rightarrow \Aut(H_\ast(X;\bbF_\ell)).$$
Since $X$ is a finite CW-complex, $H_\ast(X;\bbF_\ell)$ is finite dimensional, say of dimension $n_\ell$. Note that the homology of $X$ is independent of the CW-structure chosen, so $n_\ell$ does not depend on $n$.  Now we can view $\rho$ as a map with values in $\text{GL}(n_\ell,\bbF_\ell)$. Write $N_\ell$ for the order of $\text{GL}(n_\ell,\bbF_\ell)$, and also write $H:=\ker(\rho)$. Let $s\geq 1$ be such that $H\cong (\bbZ\slash \ell\bbZ)^s$.

Let $D$ be the dimension of $X$ as a CW-complex with respect to the CW-structure lifted through the cover $p^N:X\rightarrow X$. As mentioned above, $D$ does not depend on $n$. Since $H$ acts freely, cellularly, and trivially on homology over $\bbF_\ell$, Theorem \ref{thm:rankact} implies that $D\geq s$. Since $H$ has index at most $N_\ell$ in $G$, we find that
	\begin{align*}
	\ell^r &= |G| = [G:H] \, |H| \\
	&\leq N_\ell \, \ell^s\\
	& \leq N_\ell \, \ell^D.
	\end{align*}
Since $N_\ell$ and $D$ are independent of $n$, we obtain a uniform bound on $r$.\end{proof}

By Claim \ref{cl:rkbd}, we can choose $C\geq 1$ such that $K_n\leq C$ for every $n$. For uniformity in the description of $A_n$, let us introduce the following notation: Let $\ell_1,\dots,\ell_r$ be the (distinct) prime factors of $d$. For $1\leq i\leq r$ and every $n$, there are at most $C$ values of $k$ such that $p_{n,k}=\ell_i$. Therefore we can choose integers
	$$r_n^{(i,1)}\geq r_n^{(i,2)}\geq \dots \geq r_n^{(i,C)}\geq 0$$
such that the $\ell_i$-torsion of $A_n$ is given by
	\begin{equation} A_n \otimes \bbZ(\ell_i^\infty) \cong \bigoplus_{1\leq k\leq C} \bbZ\slash \ell_i^{r_n^{(i,k)}} \bbZ.
	\label{eq:torsion}
	\end{equation}
Note that is entirely possible that $\ell_i$ occurs strictly less than $C$ times in the decomposition \eqref{eq:decomp} of $A_n$, in which case $r_n^{(i,k)}=0$ for some values of $k$. Combining the description given by Equation \eqref{eq:torsion} of the $\ell_i$-torsion of $A_n$ for all values of $i$, we have
	$$A_n \cong \bigoplus_{1\leq i\leq r} \bigoplus_{1\leq k\leq C} \bbZ\slash \ell_i^{r_n^{(i,k)}} \bbZ.$$
It follows that $A_n$ has rank at most $rC$. It is easy to see that then $A=\cup_n A_n$ also has rank at most $rC$. Since $A$ has bounded rank, $A$ is Artinian (see Proposition \ref{prop:artincrit}).\end{proof}
	
	Since $F$ is Artinian, it is a finite extension of an abelian group (by Shunkov's and Kegel-Wehrfritz's solution of the Chernikov problem for locally finite groups, see Theorem \ref{thm:artin}). Our next lemma shows that we can modify the choice of the finite index abelian subgroup $A$ of $F$ such that $A$ is compatible with the surjection $\alpha:F\rightarrow F$.
	\begin{lemnr} Let $F$ be a locally finite Artinian group and $\alpha:F\rightarrow F$ a surjection. Then there exists a finite index abelian subgroup $A\subseteq F$ such that $\alpha(A)\subseteq A$.
	\label{lem:abinv}
	\end{lemnr}
	\begin{proof} First choose an abelian finite index subgroup $A_0$ of $F$. For $k\geq 1$, define $A_k:= \alpha(A_{k-1})\cap A_{k-1}$. Then $(A_k)_{k\geq 0}$ is a nonincreasing sequence of subgroups of $F$. Since $F$ is Artinian, this chain eventually stabilizes, say to the subgroup $A$. Note that since $\alpha$ is a surjection and $A_0$ is of finite index in $F$, the subgroup $A_k$ must also be of finite index in $F$. For $k\gg 1$ we have $A=A_k$, so $A$ is also of finite index in $F$.
	
	Finally, note that we have
		$$A=\bigcap_{k\geq 0} A_k.$$
	It immediately follows that $\alpha(A)\subseteq A$.	\end{proof}
	
We will actually show that $F$ itself is abelian, not just virtually abelian. The key observation is that $\alpha$ is \emph{locally of finite order}, i.e. for any $g\in F$, there exists $n\geq 1$ such that $\alpha^n(g)=e$. Indeed, $\alpha:F\rightarrow F$ is the map induced by the natural quotient maps
	\begin{equation} \Gamma\slash\varphi^{n+1}(\Gamma)\rightarrow \Gamma\slash\varphi^n(\Gamma), \label{eq:natquot} \end{equation}
so if $g\in F_n$, then applying the quotient maps of Equation \eqref{eq:natquot} (at most) $n$ times, we have that the image of $g$ is trivial. After passing to the limit this yields $\alpha^n(g)=e$.
 
We can now show:	\begin{lemnr} $F$ is abelian. \label{lem:ab}\end{lemnr}
\begin{proof} By Lemma \ref{lem:abinv}, there exists an abelian finite index subgroup $A\subseteq F$ with $\alpha(A)\subseteq A$. Choose finitely many coset representatives $g_1,\dots,g_m$ of $F\slash A$. Choose $N\geq 1$ such that $\alpha^N(g_i)=e$ for every $i$. Then
	$$\alpha^N(F)=\alpha^N\left(\bigcup_i g_i A\right)\subseteq A.$$
On the other hand $\alpha$ is a surjection, so we find that $F\subseteq A$. We conclude that $F=A$ is abelian. \end{proof}
	
	We will now finish the proof of Theorem \ref{thm:group2}. At this point we know (by Lemma \ref{lem:artin} and Theorem \ref{thm:artin}) that the asymptotic deck group of the strongly regular cover $p$ is abelian, and is virtually a finite product of quasicyclic groups. We need to construct a free abelian quotient of $\Gamma$ such that $\varphi$ descends to a linear endomorphism of the quotient.

Since $\Gamma\slash\varphi^n(\Gamma)$ are finite abelian groups whose orders get arbitrarily large as $n\rightarrow \infty$, we already see that $\Gamma$ has infinite abelianization and hence $\Gamma$ has a free abelian quotient. Our goal now is to construct a free abelian quotient $\Lambda$ of $\Gamma$ that is naturally associated to the map $\varphi$.

To define $\Lambda$, first note that for every $n$, we have $[\Gamma,\Gamma]\subseteq \varphi^n(\Gamma)$ (because the quotients $\Gamma\slash \varphi^n(\Gamma)$ are abelian). Define the \emph{stable image} 
	$$\varphi^\infty(\Gamma):=\bigcap_{n\geq 1} \varphi^n(\Gamma).$$
Then $[\Gamma,\Gamma]\subseteq \varphi^\infty(\Gamma)$, so $\Lambda:=\Gamma\slash \varphi^\infty(\Gamma)$ is a finitely generated abelian group. 

Since $\varphi$ restricts to an automorphism of $\varphi^\infty(\Gamma)$, we know that $\varphi$ descends to an injective morphism 
	$$\overline{\varphi}:\Lambda\rightarrow \Lambda,$$
and 	$$[\Lambda:\overline{\varphi}(\Lambda)]=[\Gamma:\varphi(\Gamma)]=d.$$
In particular $\Lambda$ is nontrivial. In addition we have:

\begin{lemnr} The stable image $\overline{\varphi}^\infty(\Lambda)=\cap_n \overline{\varphi}^n(\Lambda)$ is trivial. \label{lem:stableim}\end{lemnr}
\begin{proof} Suppose that $v\in \overline{\varphi}^\infty(\Lambda)$, and choose $v_n\in\Lambda$ with $\overline{\varphi}^n(v_n)=v$. Let $\gamma\in\Gamma$ map to $v_n\in\Lambda=\Gamma\slash\varphi^\infty(\Gamma)$. Similarly choose $\gamma_n\in \Gamma$ that map to $v_n$. 

We will show that $\gamma\in\varphi^\infty(\Gamma)$, so that $v=0$. Let $n\geq 1$ be arbitrary. Since $\gamma$ and $\varphi^n(\gamma_n)$ both map to $v\in \Lambda$, we know that $\gamma \varphi^n(\gamma_n^{-1})\in \varphi^\infty(\Gamma)$. Since $\varphi$ restricts to an automorphism on $\varphi^\infty(\Gamma)$, there exist $\delta_n\in\Gamma$ with
	$$\gamma \varphi^n(\gamma_n^{-1})=\varphi^n(\delta_n).$$
We may rewrite this as
	$$\gamma=\varphi^n(\delta_n)\varphi^n(\gamma_n)=\varphi^n(\delta_n\gamma_n),$$
so $\gamma\in \varphi^n(\Gamma)$. Since $n$ was arbitrary, we have $\gamma\in\varphi^\infty(\Gamma)$, as desired. \end{proof}

It remains to show that $\Lambda$ is torsion-free. Let $\Delta$ be the torsion subgroup of $\Lambda$. Since $\Lambda$ is finitely generated, $\Delta$ is finite. Since $\overline{\varphi}$ is injective and the image of a torsion element under a homomorphism is still torsion, we see that $\overline{\varphi}$ restricts to an automorphism of $\Delta$. Therefore $\Delta$ is contained in the stable image of $\overline{\varphi}$, which we have just shown to be trivial. Therefore $\Delta$ is trivial, so that $\Lambda$ is torsion-free. This completes the proof of Theorem \ref{thm:group2}.

\section{Smoothly proper self-covers and principal torus bundles}
\label{sec:proper}

The goal of this section is to prove Theorem \ref{thm:proper}, a topological classification theorem of smoothly proper, strongly regular self-covers of closed manifolds. Namely, in this case we show that a finite cover of the manifold is a principal torus bundle, and the self-cover lifts to
 a bundle map that is a cover on each fiber and descends to a diffeomorphism on the base. 

For the rest of this section, let $M$ be a closed manifold and let $p:M\rightarrow M$ be a nontrivial, smoothly proper self-cover with asymptotic deck group $F$. We will first investigate the compact group $\overline{F}$. Recall that there is a surjection $\alpha:F\rightarrow F$ induced by the natural quotient maps
	$$\Gamma\slash\varphi^{n+1}(\Gamma)\rightarrow \Gamma\slash\varphi^n(\Gamma),$$
and that $p$ is $\alpha$-equivariant (see Lemma \ref{lem:aeq}). We can now get some information about the compact group $\overline{F}$:

\begin{lemnr} $\overline{F}$ is a compact connected Lie group with finitely many components. Its identity component $\overline{F}^0$ is isomorphic to a torus. \label{lem:lie}\end{lemnr}

\begin{proof} The Hilbert-Smith conjecture asserts that a compact group of homeomorphisms of a topological manifold is a Lie group. It has been proven for compact groups of diffeomorphisms by Montgomery \cite{smoothhs}. Therefore $\overline{F}$ is Lie.

It remains to show $\overline{F}^0$ is isomorphic to a torus. Note that $\overline{F}^0$ is abelian, because $F\cap \overline{F}^0$ is a dense abelian subgroup. Therefore $\overline{F}^0$ is a compact connected abelian Lie group, and hence it must be isomorphic to a torus.\end{proof}

Let $\widehat{M}$ be the cover of $M$ corresponding to the subgroup $\varphi^\infty(\Gamma)\subseteq \Gamma$. Then $\widehat{M}\rightarrow M$ is a regular cover with deck group $\Lambda:=\Gamma\slash\varphi^\infty(\Gamma)$. By Theorem \ref{thm:group}, $\Lambda$ is free abelian.

We also know the map $\varphi:\Gamma\hookrightarrow \Gamma$ descends to a linear endomorphism $\varphi$ of $\Lambda$ with determinant $d$. Since $\varphi$ restricts to an isomorphism of $\varphi^\infty(\Gamma)$, the self-cover $p:M\rightarrow M$ lifts to a diffeomorphism $\hat{p}:\widehat{M}\rightarrow \widehat{M}$ that is $\overline{\varphi}$-equivariant.

Finally, let $G$ be the group of maps $\widehat{M}\rightarrow \widehat{M}$ that are lifts of elements of $\overline{F}$. Then $G$ is an extension
	$$1\rightarrow \Lambda\rightarrow G\rightarrow \overline{F}\rightarrow 1.$$
Let $D$ be the preimage of $F$ in $G$. First, we describe how $G$ and $\hat{p}$ are related.

\begin{lemnr} Conjugation by $\hat{p}$ induces an automorphism of $G$ that restricts to $\overline{\varphi}$ on $\Lambda$. \label{lem:aut}\end{lemnr}
	\begin{proof} The action of $F_n$ on $M$ is defined by
		$$g\cdot x:=p_n^{-1} g p_n(x)$$
	where $p_n$ are diffeomorphisms $M\rightarrow \widetilde{M}\slash\varphi^n(\Gamma)$ that lift $p^n:M\rightarrow M$ (see Lemma \ref{lem:extact}). Since $p_n$ lifts to $\hat{p}^n$, it is immediate that the group of lifts of $F_n$ is given by $\hat{p}^{-n} \Lambda \hat{p}^n$, and 
	$$D=\bigcup_{n\geq 0} \hat{p}^{-n} \Lambda\, \hat{p}^n.$$
Since $\hat{p}$ is $\overline{\varphi}$-equivariant, we have for any $g\in \Lambda$:
	$$\hat{p}g\hat{p}^{-1}=\overline{\varphi}(g)\hat{p}\hat{p}^{-1}=\overline{\varphi}(g).$$
So conjugation by $\hat{p}$ restricts to $\overline{\varphi}:\Lambda\hookrightarrow \Lambda$. Therefore conjugation by $\hat{p}$ restricts to an automorphism of $D$. Since $D$ is dense in $G$, it follows that conjugation by $\hat{p}$ is an automorphism of $G$.\end{proof}
	
Hence we can extend $\overline{\varphi}:\Lambda\hookrightarrow \Lambda$ to an automorphism of $G$ by setting $\overline{\varphi}(g):=\overline{p}g\overline{p}^{-1}$. Then
		$$D=\bigcup_{n\geq 0} \overline{\varphi}^{-n}(\Lambda).$$
Since $\overline{\varphi}(\Lambda)\subseteq \Lambda$, we know that $\overline{\varphi}$ descends to a map of $G\slash \Lambda\cong \overline{F}$. Write $\overline{\alpha}:\overline{F}\rightarrow \overline{F}$ for this map.
	
\begin{lemnr} $\overline{\alpha}$ extends $\alpha:F\rightarrow F$. Further $p:M\rightarrow M$ is $\overline{\alpha}$-equivariant. \label{lem:aext}\end{lemnr}
\begin{proof} Let $g\in F$. By Lemma \ref{lem:aeq}, we know that $p$ is $\alpha$-equivariant, so we have
	\begin{equation} p\circ g = \alpha(g)\circ p 
	\label{eq:peq} \end{equation}
(as diffeomorphisms of $M$). Choose a lift $\hat{g}$ of $g$ to $\hat{M}$. By Equation \eqref{eq:peq}, there is a lift $\widehat{\alpha(g)}$ of $\alpha(g)$ to $\widehat{M}$ such that
	$$\hat{p}\circ \hat{g}=\widehat{\alpha(g)}\circ \hat{p}$$
so that $\widehat{\alpha(g)}=\hat{p}\hat{g}\hat{p}^{-1}$. For the maps on $M$ this means
	$$\alpha(g)=\overline{\alpha}(g)$$
as desired.

To see that $p$ is $\overline{\alpha}$-equivariant, note that we already know that $p$ is $\alpha$-equivariant, and $F$ is dense in $\overline{F}$. Since $\overline{\alpha}$ extends $\alpha$, it immediately follows that $p$ is $\overline{\alpha}$-equivariant as well.\end{proof}

\begin{lemnr} $\overline{F}$ is connected. \label{lem:fconn}\end{lemnr}
	
\begin{proof} Let $G^0$ be the connected component of $G$ that contains the identity. Since $\hat{p}(G^0)=G^0$ and $G^0$ surjects onto $\overline{F}^0$, we know that $\overline{\alpha}(\overline{F}^0)=\overline{F}^0$. Therefore $\overline{\alpha}$ descends to a map of the group of components $\pi_0(F)\cong \overline{F}\slash\overline{F}^0$ of $\overline{F}$.
	
$\pi_0(F)$ is finite because $\overline{F}$ is compact. On the other hand, $\overline{\alpha}$ is a surjection (since $\overline{F}$ is compact, $\alpha$ is surjective, and $F$ is dense in $\overline{F}$). Therefore the induced map on components is a surjection, and because $\pi_0(F)$ is finite, $\overline{\alpha}$ must induce an isomorphism on $\pi_0(F)$.

We will again use the principle that $\alpha$ being locally of finite order can be used to show that a finite index subgroup (in this case $\overline{F}^0$)	is the entire group (in this case $\overline{F}$): 
Take coset representatives $g_1,\dots,g_r\in \overline{F}$ for the cosets of $\overline{F}\slash\overline{F}^0$. Since $F$ is dense, $F$ surjects onto $\overline{F}\slash\overline{F}^0$. Therefore we can choose $g_i\in F$. Since $\alpha$ is locally of finite order, we can choose $N\gg 1$ such that $\alpha^N(g_i)=e$. Then we have
	$$\overline{\alpha}^N(\overline{F})=\overline{\alpha}^N\left(\bigcup_i g_i \overline{F}^0\right)=\overline{F}^0,$$
so $\overline{F}$ is connected.\end{proof}

We will suspend our investigation of $G$ for the moment, because we need some more precise algebraic information about the relationship between $\Lambda$ and $D=\cup_n \overline{\varphi}^{-n}(\Lambda)$. 

Consider the vector group $V:=\Lambda\otimes\bbR$. Since $\overline{\varphi}$ is a linear endomorphism of $\Lambda$, it naturally extends to some linear transformation
		$$\overline{\varphi}:V\rightarrow V.$$
This allows us to define a natural embedding $j:D\hookrightarrow V$ by setting 
	\begin{equation} j(\overline{\varphi}^{-n}(g)):=\overline{\varphi}^{-n}(g\otimes 1) \label{eq:jdef}\end{equation}
for $g\in \Lambda$ and $n\geq 0$. It is straightforward to check that $j$ is well-defined.

\begin{lemnr} $j(D)$ is dense in $V$. \label{lem:jdense} \end{lemnr}

\begin{proof} The closure $\overline{j(D)}$ is a closed connected subgroup of $V$. Closed subgroups of connected abelian Lie groups have been classified by W\"ustner \cite{abelianlie}. In this case, because $V$ has no compact subgroups, this classification yields a discrete subgroup $\Delta\subseteq V$ and a subspace $W\subseteq V$ such that $\overline{j(D)}=\Delta\oplus W$. Therefore it suffices to show that $\Delta=0$. 

Note that while $W=\overline{j(D)}^0$ is uniquely determined by $j(D)$, the discrete group $\Delta$ is not. Indeed, if $W$ is nontrivial then there are many sections of 
	$$\overline{j(D)}\rightarrow \overline{j(D)}\slash W \cong \Delta$$
and the image of any such section would be an admissible choice of $\Delta$. We will use this freedom to make a particular choice of $\Delta$ that is compatible with the lattice $\Lambda\subseteq \overline{j(D)}$.

Write $\Lambda_\Delta:=\Lambda\cap \Delta$ and $\Lambda_W:=\Lambda\cap W$. Since $\Lambda\subseteq \overline{j(D)}$ is a lattice, we have
	\begin{enumerate}
		\item $\Lambda_W$ is a lattice in $W=\overline{j(D)}^0$, and
		\item The composition
	$$\Lambda\hookrightarrow \overline{j(D)}=\Delta\oplus W\rightarrow \Delta$$
has image that is of finite index in $\Delta$, and is a cocompact lattice in $V\slash W$.
	\end{enumerate}
Now consider the short exact sequence
	$$1\rightarrow \Lambda_W \rightarrow \Lambda \rightarrow \Lambda\slash \Lambda_W\rightarrow 1.$$
Since all of these groups are free abelian, this sequence splits. Choose any section $\sigma:\Lambda\slash \Lambda_W\rightarrow \Lambda$. Since $\Lambda\slash \Lambda_W$ is a cocompact lattice in $V\slash W\cong \Delta\otimes \bbR$, we can uniquely extend $\sigma$ to a section $V\slash W\rightarrow V$. Set $\Delta:=\sigma\left(\overline{j(D)}\slash W\right)$. 

Now we have $\Lambda=\Lambda_\Delta\oplus \Lambda_W$, and we have already seen that
	\begin{equation} V=\Lambda\otimes \bbR = (\Delta\otimes \bbR)\oplus W. \label{eq:vdecomp} \end{equation}
 Consider the linear transformation $\overline{\varphi}$. Since $\overline{\varphi}$ restricts to an automorphism of $j(D)$, we know that $\overline{\varphi}$ must preserve $W$, and hence is triangular with respect to the decomposition of Equation \eqref{eq:vdecomp}, say
	\begin{equation} \overline{\varphi}=\begin{pmatrix} \overline{\varphi}_\Delta & 0 \\ C & \overline{\varphi}_W\end{pmatrix} \label{eq:triang}\end{equation}
where $\overline{\varphi}_\Delta$ (resp. $\overline{\varphi}_W$) is an automorphism of $\Delta\otimes\bbR$ (resp. $W$), and 
	$$C:\Delta\otimes\bbR\rightarrow W$$ 
is a linear map with $C(\Lambda_\Delta)\subseteq \Lambda_W$.

Set $N:=[\Delta:\Lambda_\Delta]$. Since $\Delta$ is finitely generated, it has only finitely many subgroups of index $N$. Let $\Delta_N$ be the intersection of all subgroups of $\Delta$ of index $N$. Then $\Delta_N\subseteq \Lambda_\Delta$ is a finite index subgroup and and $\overline{\varphi}_\Delta$ restricts to an automorphism on $\Delta_N$.

Our goal is to show that if $\Delta$ is nontrivial, we can modify $\Delta$ even further so that it is also compatible with $\overline{\varphi}$. Namely, we want to show that we can choose $\Delta$ such that $\overline{\varphi}(\Delta)\subseteq \Delta$. For this, we need to block diagonalize the block triangular matrix of Equation \eqref{eq:triang}. For this, we need:

\begin{claimnr} Let $\lambda$ be an eigenvalue of $\overline{\varphi}_\Delta$. Then $\lambda$ is not an eigenvalue of $\overline{\varphi}_W$. \label{claim:eval}\end{claimnr}
\begin{proof} $\Lambda_N$ is a cocompact lattice in $V$, so that we can choose an isomorphism $V\cong \bbR^{\dim(V)}$ such that $\Lambda_N$ corresponds to $\bbZ^{\dim(V)}$. Since $\overline{\varphi}(\Lambda_N)\subseteq \Lambda_N$, the matrix corresponding to $\overline{\varphi}$ has integer entries. For $\beta\in \bbC$ an algebraic integer, let $K_\beta$ be the splitting field of the minimal polynomial of $\beta$.


Now let $\lambda$ be an eigenvalue of $\overline{\varphi}_\Delta$. We argue by contradiction, so suppose that $\lambda$ is also an eigenvalue of $\overline{\varphi}_W$. First consider $\overline{\varphi}_\Delta$: Since $\overline{\varphi}_\Delta$ is an integer matrix, $\lambda$ is an algebraic integer and the Galois conjugate of $\lambda$ by any $\sigma\in \text{Gal}(K_\lambda\slash\bbQ)$ is also an eigenvalue of $\overline{\varphi}_\Delta$. Therefore the algebraic norm 
	$$N_{K_\lambda\slash\bbQ}(\lambda)=\prod_{\sigma\in\text{Gal}(K_\lambda\slash\bbQ)} \sigma(\lambda)$$
divides $\det(\overline{\varphi}_\Delta)=\pm 1$. On the other hand $N_{K_\lambda\slash\bbQ}(\lambda)$ is an integer (because $\lambda$ is an algebraic integer), and hence we must have $N_{K_\lambda\slash\bbQ}(\lambda)=\pm 1$. 

Now consider $\overline{\varphi}_W$: First extend $\overline{\varphi}_W$ to a map of the complexification $W_\bbC:=W\otimes_{\bbR} \bbC$. Define the subspace
	$$U:=\bigoplus_{\sigma \in \text{Gal}(K_\lambda \slash \bbQ)} \ker(\overline{\varphi}_W-\sigma(\lambda)\text{Id}).$$
of $W_\bbC$. Then $U$ is an $\overline{\varphi}_W$-invariant subspace of $W_\bbC$. Write $\overline{\varphi}_U$ for the restriction of $\overline{\varphi}_W$ to $U$. Then $\overline{\varphi}_U$ has eigenvalues $\sigma(\lambda)$, where $\sigma\in \text{Gal}(K_\lambda\slash\bbQ)$, and each eigenvalue occurs with the same algebraic multiplicity $m\geq 1$. Hence 
	\begin{align*}\det(\overline{\varphi}_U)	&=\prod_{\sigma\in\text{Gal}(K_\lambda\slash \bbQ)} \sigma(\lambda)^m 	\\
									&= \left(N_{K_\lambda\slash \bbQ}(\lambda)\right)^m					\\
									&=\pm 1,
	\end{align*}
Further, because $\ker(\overline{\varphi}_W-\lambda I)$ is defined over $\bbQ(\lambda)$ and $U$ is spanned by all Galois conjugates of this eigenspace, we know that $U$ is defined over $\bbQ$.

Set $\Lambda_U:=\Lambda_N\cap U$. Then $\Lambda_U$ is a cocompact lattice in $U(\bbR)$ because $U$ is defined over $\bbQ$ (hence over $\bbR$). Since $\overline{\varphi}$ restricts to a linear automorphism of $U$ with determinant $\pm 1$, and $\overline{\varphi}(\Lambda_U)\subseteq \Lambda_U$, we must have that $\overline{\varphi}_U$ restricts to an automorphism of $\Lambda_U$. Therefore $\Lambda_U$ is contained in the stable image
	$$\overline{\varphi}^\infty(\Lambda):=\cap_n \overline{\varphi}^n(\Lambda)$$
which is trivial by Lemma \ref{lem:stableim}. On the other hand, $\Lambda_U\subseteq U(\bbR)$ is a cocompact lattice, so $U(\bbR)$ must be trivial. This is a contradiction, and completes the proof of Claim \ref{claim:eval}. \end{proof}

Let $\lambda_1,\dots,\lambda_r$ be the eigenvalues of $\overline{\varphi}_\Delta$ (counted without multiplicity), and let
	$$E_i:=\ker(\overline{\varphi}-\lambda_i \text{Id})$$
be the corresponding generalized eigenspace in the complexification $V_\bbC$. As discussed above, $U:=\oplus_i E_i$ is defined over $\bbQ$, and hence $\Lambda_N\cap U(\bbR)$ intersects $U(\bbR)$ in a cocompact lattice. In addition, $U(\bbR)$ is clearly $\overline{\varphi}$-invarinat, so
	$$\overline{\varphi}(A_N\cap U(\bbR))\subseteq A_N\cap U(\bbR).$$
We want to show that we must have equality. To do so, it suffices to show that $\det(\overline{\varphi}|)_{U(\bbR)})=\pm 1$: Since none of the eigenvalues $\lambda_i$ are eigenvalues of $\overline{\varphi}_W$ (by Claim \ref{claim:eval}), we know that $U(\bbR)$ intersects $W$ trivially and hence $U(\bbR)$ projects isomorphically onto $V\slash W$. In particular 
	$$\det\overline{\varphi}|_{U(\bbR)}=\det\overline{\varphi}_\Delta=\pm 1.$$ 
We conclude that $\overline{\varphi}(A_N\cap U(\bbR))=A_N\cap U(\bbR)$. But then $\Lambda_N\cap U(\bbR)$ is contained in the stable image of $\overline{\varphi}$, which is a contradiction. Therefore we must have that $\Delta=0$ is trivial, as desired. This completes the proof of Lemma \ref{lem:jdense}.\end{proof}

We can now continue our investigation of $G$.
\begin{lemnr} $G$ has finitely many components. \label{lem:vconn} \end{lemnr}
	
\begin{proof} The group of components of $G$ is $\pi_0 G \cong G\slash G^0$. Since $\Lambda$ is abelian, $D=\cup_n \overline{\varphi}^{-n}(\Lambda)$ is the increasing union of abelian groups and hence abelian. Since $D$ is dense in $G$, we find that $G$ is abelian as well. Therefore $\pi_0(G)$ is an abelian group. In fact it must also be finitely generated: 

Since $\Lambda$ is a cocompact lattice in $G$, the image of $\Lambda$ under
		$$\Lambda\hookrightarrow G\rightarrow G\slash G^0$$
is of finite index in $G\slash G^0$. Therefore $G\slash G^0$ contains a finite index subgroup that is finitely generated and abelian. Since $G\slash G^0$ is also abelian itself (because $G$ is abelian), we see that $G\slash G^0$ is a finitely generated abelian group. We need to show that $G\slash G^0$ is finite, or equivalently, that there is no nontrivial homomorphism $G\slash G^0\rightarrow \bbZ$.
	
Suppose that there is a homomorphism $G\slash G^0\rightarrow \bbZ$. By precomposing with the projection $G\rightarrow G\slash G^0$, we obtain a homomorphism
		$$\lambda:G\rightarrow \bbZ.$$
Since $\Lambda$ is free abelian, the restriction of $\lambda$ to $\Lambda$ naturally extends to the vector group $V:=\Lambda\otimes\bbR$ and gives a linear functional
		$$f:V\rightarrow \bbR$$
that is integer-valued on $\Lambda$. Further, since $\overline{\varphi}$ is a linear endomorphism of $\Lambda$, it also naturally extends to some linear transformation
		$$\overline{\varphi}:V\rightarrow V.$$
Let $j:D\hookrightarrow V$ be the embedding of Equation \eqref{eq:jdef}. Note that we now have two functionals on $D$, namely $\lambda$ and $f\circ j$. These two functionals coincide on $\Lambda$. By linearity, and since any element of $D$ has a multiple that belongs to $\Lambda$, we must have
	$$\lambda=f\circ j$$
as maps $D\rightarrow \bbR$. In particular $f$ is integer-valued on $j(D)$. On the other hand, by Lemma \ref{lem:jdense}, $j(D)$ is dense in $V$. Therefore $f$ is integer-valued on all of $V$. But a linear, integer-valued functional on a real vector space is trivial. Hence the original homomorphism $G\slash G^0\rightarrow \bbZ$ is trivial, as desired.\end{proof}

Before proceeding we need the following technical statement:
	\begin{lemnr} Let $H$ be a connected abelian Lie group and let $\Delta\subseteq H$ be a torsion-free cocompact subgroup. Suppose $\psi:H\rightarrow H$ is an automorphism with $\psi(\Delta)\subseteq \Delta$ and such that $D:=\displaystyle\bigcup_{n\geq 0} \psi^{-n}(\Delta)$ is dense in $H$. Then $H$ is a vector group. \label{lem:tech} \end{lemnr}
	\begin{proof} We proceed by induction on $\dim H$. The base case $\dim(H)=0$ is trivial.
	
	Suppose now that the statement is true for groups of dimension $<N$ for some $N\geq 1$. For $n\geq 0$, set $\Delta_n:=\psi^{-n}(\Delta)$. Choose a sequence of open neighborhoods $U_k$ of $e\in H$ such that the exponential map of $H$ is a diffeomorphism onto $U_k$ and $U_{k+1}\subseteq U_k$ and $\psi(U_{k+1})\subseteq U_k$. For fixed $k$, consider the sequence
		$$V^{(k)}_n:=\text{span}\left(\log(\Delta_n\cap U_k)\right)$$
	of subspaces of the Lie algebra $\frh$ of $H$. Since $\Delta_n\subseteq \Delta_{n+1}$, we have $V^{(k)}_n\subseteq V^{(k)}_{n+1}$. Since $\frh$ is finite dimensional, this sequence stabilizes to some subspace $V^{(k)}\subseteq \frh$. Since $U_{k+1}\subseteq U_k$, we have $V^{(k+1)}\subseteq V^{(k)}$. Therefore the sequence $V^{(k)}$ stabilizes to some subspace $V$. Further since $\psi(U_{k+1})\subseteq U_k$, we have $\psi_\ast(V)\subseteq V$. Since $\psi$ is injective and $V$ is finite dimensional, $\psi$ restricts to an automorphism of $W:=\exp(V)$. 
	
	We claim that $W$ is a nontrivial closed vector subgroup of $H$. Indeed, because $D=\cup_n \Delta_n$ is dense, we know that for $n$ and $k$ sufficiently large, $\Delta_n\cap U_k\neq \emptyset$. Therefore $V\neq 0$. This shows that $W$ is nontrivial. Further, again for $n$ and $k$ sufficiently large, $\text{span}\log(\Delta_n\cap U_k)$ is a cocompact lattice in $V$. Since $H$ is abelian, 
		$$\exp:\frh \to H$$
is a morphism, so $\exp \text{span}\log(\Delta_n\cap U_k)$ is a subgroup of $\Delta_n$, which is discrete. This shows that
		$$\exp: V\rightarrow H$$
	is a proper map, and hence $W=\exp(V)$ is closed and in addition $W$ is a vector subgroup.
	
	We consider two cases: Either $\Delta$ is contained in $W$ or not. First suppose $\Delta$ is contained in $W$. Since $\psi$ restricts to an automorphism of $W$, we have $D\subseteq W$. Since $D$ is dense and $W$ is closed, we find that $W=H$, and hence $H$ is a vector group.
	
	Finally suppose $\Delta$ is not contained in $W$. Consider the projection map $q_W: H\rightarrow H\slash W$. Since $W$ is a nontrivial closed vector subgroup, we see that $H\slash W$ is a connected abelian Lie group with $\dim(H\slash W)<\dim H$. We want to apply the inductive hypothesis, so we verify
		\begin{itemize}
			\item Since $\Delta\cap W$ is a cocompact lattice in $W$, the image of $\Delta$ in $H\slash W$ is also a cocompact lattice,
			\item $\psi$ restricts to an automorphism of $W$, and hence descends to some automorphism $\overline{\psi}$ of $H\slash W$, and
			\item 
		$$\bigcup_{n\geq 0} \overline{\psi}^{-n}(\Delta\slash (\Delta\cap W)) = q_W(D)$$
	is dense in $H\slash W$ because $D$ is dense in $H$.
		\end{itemize}
All that prevents us from applying the inductive hypothesis is that $\Delta$ may not be torsion-free. But $\Delta$ is a finitely generated abelian group, so its torsion subgroup Tor$(\Delta)$ is a finite normal subgroup. It is $\overline{\psi}$-invariant, and since $\overline{\psi}$ is injective and Tor$(\Delta)$ is finite, $\overline{\psi}$ restricts to an automorphism of Tor($\Delta$). This allows us to consider $(H\slash W)\slash \text{Tor}(\Delta)$ instead of $H\slash W$. The map 
	$$H\slash W\rightarrow (H\slash W)\slash\text{Tor}(\Delta)$$
is a finite degree cover, so one easily verifies from the above that $\Delta$ is a cocompact lattice, and the image of $\Delta$ is now torsion-free. Further $\overline{\psi}$ descends to an automorphism and the image of $D$ is dense. 
	
	By the inductive hypothesis applied to $H\slash W\slash \text{Tor}(\Delta)$, we know that $H\slash W\slash\text{Tor}(\Delta)$ is a vector group. The map $H\slash W\rightarrow H\slash W\slash\text{Tor}(\Delta)$ is a finite connected covering of the vector group $H\slash W\slash\text{Tor}(\Delta)$, and hence an isomorphism. Therefore $H\slash W$ is a vector group. Thus $H$ is an extension of vector groups
		$$1\rightarrow W\rightarrow H\rightarrow H\slash W\rightarrow 1.$$
	Since $H$ is abelian, this sequence splits. Therefore $H$ is also a vector group.\end{proof}
	
Lemma \ref{lem:tech} readily applies to our situation and shows that $G^0$ is a vector group. This finishes our investigation of $G$, and we return to studying $M$:	
	
\begin{lemnr} There is a finite cover $M'\to M$ such that 
		\begin{enumerate}[(i)]
			\item $p:M\to M$ lifts to a strongly regular self-cover $p':M'\to M'$, and
			\item Let $F'$ be the asymptotic deck group of $p'$. Then $\overline{F'}$ is a compact connected torus that acts freely on $M'$.
		\end{enumerate}
		\label{lem:cover}
\end{lemnr}

\begin{proof} As above, write $\Lambda_0:=\Lambda\cap G^0$. Since $G$ has finitely many components by Lemma \ref{lem:vconn}, $\Lambda_0$ is of finite index in $\Lambda$. Hence $M':=\widehat{M}\slash \Lambda_0$ is a finite cover of $M$. We claim that $M'$ has the required properties. First let us show that $p:M\rightarrow M$ lifts to a strongly cover $p':M'\rightarrow M'$. The lift $p'$ exists precisely when $\overline{\varphi}(\Lambda_0)\subseteq \Lambda_0$. This holds because $\overline{\varphi}$ is given by conjugation by $\hat{p}$, and hence preserves $G^0$. Further note that $p'$ is also strongly regular because $\Lambda_0$ is abelian, so that $\overline{\varphi}^n(\Lambda_0)$ is automatically normal in $\Lambda_0$.

It remains to analyze the asymptotic deck group $F'$ of $p'$. Write $D_0:=D\cap G^0$. Then we have $F'=D_0\slash \Lambda_0$, and $\overline{F'}=G^0\slash \Lambda_0$. In particular, $\overline{F'}$ is a connected torus. We need to show that $\overline{F'}$ acts freely on $M'$.

Let $x_0\in M$ and choose $\overline{x}_0\in \overline{M}$ that projects to $x_0$. Let $S$ be the stabilizer in $\overline{F'}$ of $x_0$. Because $S$ fixes a point, the $S$-action lifts to an $S$-action on $\widehat{M}$ fixing $\overline{x}_0$ (even if $S$ is not connected!). The group of lifts of $\overline{F'}$ to $\widehat{M}$ is exactly $G^0$, so we have an embedding $S\hookrightarrow G^0$. But since $G^0$ is a vector group, it has no nontrivial compact subgroups, so we see that $S$ is trivial, as desired.\end{proof}	
	
Armed with Lemma \ref{lem:cover}, we can finish the proof of Theorem \ref{thm:proper}. 

\begin{proof}[Proof of Theorem \ref{thm:proper}] Let $p':M'\rightarrow M$ be the cover produced by Lemma \ref{lem:cover}. Let $p':M'\to M'$ be a lift of $p$ and set $T:=\overline{F'}$. Since $T$ acts smoothly and freely on $M'$, the map natural quotient map $M'\to M'\slash T$ is a smooth principal $T$-bundle. By Lemma \ref{lem:aext}, applied to $p'$, we have a surjective map $\overline{\alpha}:T\to T$ such that $p'$ is $\overline{\alpha}$-equivariant, i.e. for any $x\in M'$ and $g\in T$, we have
	$$p'(gx)=\overline{\alpha}(g)p'(x).$$
So we see that $p'$ is a bundle map, and identifying the $T$-orbits of $x$ and $p(x)$ with $T$, the restriction of $p'$ to a $T$-orbit is exactly given by $\overline{\alpha}$. The map $\overline{\alpha}$ is linear: Indeed, $\overline{\alpha}$ by its very definition lifts to the map $\overline{\varphi}:G^0\rightarrow G^0$. As noted above, Lemma \ref{lem:tech} implies that $G^0$ is a vector group, and $\overline{\varphi}$ is a linear map with determinant $d$. Therefore $p'$ restricts to a linear cover of degree $d$ on each $T$-orbit (which can be identified with $T=G^0\slash \Lambda_0$).

It only remains to show that we have $\pi_1(M')\cong \pi_1(T)\times \pi_1(M'\slash T)$. To obtain this splitting, we consider the long exact sequence on homotopy groups for the fiber bundle $M'\rightarrow M'\slash T$. Because $\pi_2(T)=0$, the end of this long exact sequence reduces to
	$$1\rightarrow \pi_1(T)\rightarrow \pi_1(M')\rightarrow \pi_1(M'\slash T)\rightarrow 1.$$
We claim this sequence splits trivially. To show this, we produce a left splitting $\pi_1(M')\rightarrow \pi_1(T)$. Note that since $T=G^0\slash \Lambda_0$, we have $\pi_1(T)=\Lambda_0$. The left splitting is then just the map $q: \pi_1(M')\rightarrow \Lambda_0$ produced by Theorem \ref{thm:group}. This completes the proof of Theorem \ref{thm:proper}.\end{proof}

\section{Holomorphic endomorphisms of K\"ahler manifolds}
\label{sec:kaehler}

Suppose $M$ is a K\"ahler manifold and $p:M\rightarrow M$ is a holomorphic strongly regular nontrivial self-cover. 

\begin{lemnr} $p$ is smoothly proper. \end{lemnr}
\begin{proof} Note that the asymptotic deck group $F$ of $p$ acts by biholomorphisms. Boothby-Kobayashi-Wang \cite{holaut} proved that the group of biholomorphisms of a closed almost complex manifold is a Lie group (possibly with infinitely many components). Let Aut$(M)$ denote the group of biholomorphisms of $M$ and let Aut$^0(M)$ denote the connected component of Aut$(M)$ that contains the identity element. Fujiki and Liebermann independently proved that Aut$^0(M)$ is a subgroup of finite index in the kernel of the action of Aut$(M)$ on homology \cite{fujiki, lieberman}.

Note that the action of $F$ on homology factors through a finite group: Namely, the action on homology is given by a map $F\rightarrow GL(H_\ast(M,\bbZ))$. Since $H_\ast(M,\bbZ)$ has finite rank, there is a bound on the order of torsion elements of GL($H_\ast(M,\bbZ))$. It follows that any subgroup of GL($H_\ast(M,\bbZ))$ every one of whose elements has arbitrarily large roots, is finite. Since $F$ is a finite product of finite cyclic and quasicyclic groups, $F$ is generated by elements with arbitrarily large roots. Hence the image of $F$ in GL$(H_\ast(M,\bbZ))$ is finite.

It follows that $F_0:=F\cap \Aut^0(M)$ is of finite index in $F$. Let $H:=\overline{F_0}$ be the closure of $F_0$ in $\Aut^0(M)$. Since $\Aut(M)$ is a closed subgroup of Diff$(M)$, it follows that $H$ is also the closure of $F_0$ in Diff$(M)$. Therefore to show that $p$ is proper, it suffices to show that $H$ is compact. Note that $H$ is an abelian Lie group (not necessarily connected). We will again use the classification by W\"ustner of closed abelian subgroups of connected Lie groups \cite{abelianlie}. In particular he showed that the torsion subgroup of $\pi_0 H$ is finite. We conclude that $H$ has finitely many components. But $\pi_0 H$ is torsion (because $F_0\subseteq H$ is dense and torsion), so $\pi_0 H$ is finite.

Let $H^0$ be the connected component of the identity of $H$, and set $F_{00}:=F_0\cap H^0$. Since $H^0$ is a connected abelian Lie group, there are nonnegative integers $k,l\geq 0$ such that $H^0\cong T^k\times \bbR^l$. Since $F_{00}$ is dense in $H^0$, the image of $F_{00}$ in $H^0\slash T^k\cong \bbR^l$ is dense. On the other hand $F_{00}$ is torsion and $\bbR^l$ is torsion-free, so the image of $F_{00}$ in $\bbR^l$ must be trivial. Hence we must have that $l=0$, so that $H^0$ is compact. Since $H$ has finitely many components, it follows that $H$ is compact, as desired.\end{proof}

We finish the proof of Theorem \ref{thm:kaehler}. Since $p$ is smoothly proper, Theorem \ref{thm:proper} shows that there is a finite cover $M'\rightarrow M$ such that $p$ lifts to a smoothly proper, strongly regular self-cover $p':M'\rightarrow M'$, and $M'$ is a smooth principal $T$-bundle $M'\rightarrow N$. Here $T=\overline{F'}$ is the closure of the asymptotic deck group of $p'$.
Since $M$ is K\"ahler and $p$ is holomorphic, $F'$ acts on $M'$ by biholomorphisms. Since the biholomorphism group is closed in the group of all diffeomorphisms, $T$ also acts on $M'$ by biholomorphisms, which shows that the bundle $M'\rightarrow N$ is holomorphic.

As in the previous section, we know that $p'$ descends to a diffeomorphism of $N$. Since $p'$ is holomorphic and $M'\rightarrow N$ is a holomorphic bundle, the induced diffeomorphism of $N$ is also holomorphic. In addition, on each $T$-orbit, $p$ restricts to a holomorphic automorphism of degree $d$.

It remains to show that $N$ is K\"ahler and the bundle $M'\rightarrow N$ is trivial. For this we need results of Blanchard \cite{kaehlerfiber}. He showed that if the the total space of a holomorphic fiber bundle is K\"ahler, then so is the base. In addition he gives criteria for the bundle to admit a flat holomorphic connection. In the case of a holomorphic principal torus bundle with K\"ahler total space, the results of Blanchard can be used to show the existence of a flat holomorphic connection; a short independent proof of the existence of a flat connection for K\"ahler principal torus bundles has been given by Biswas \cite{kaehlerflat}. 

The principal flat $T$-bundle $M'\rightarrow N$ is characterized by its monodromy, which is a representation
	$$\rho: \pi_1(N)\rightarrow T$$
where we write $T=V\slash \overline{\Lambda}$ for a complex vector space $V$ and a cocompact lattice $\overline{\Lambda}$. Recall that two flat principal $T$-bundles over $N$ with monodromies $\rho_0$, $\rho_1$ are equivalent if and only if $\rho_0$ and $\rho_1$ are isotopic, i.e. $\rho_0$ and $\rho_1$ belong to the same path-component of Hom($\pi_1(N),T$). Our goal is to show that after passing to a further finite cover, $\rho$ is isotopic to the trivial map.

Since $\rho$ takes values in the abelian group $T$, we have a factorization of $\rho$ through the abelianization $\Delta:=\pi_1(N)^{\text{ab}}$ of $\pi_1(N)$. Since $\Delta$ is a finitely generated abelian group, we can write
	$$\Delta=\text{Tor}(\Delta)\oplus \Delta_\infty$$
where $\Delta_\infty\subseteq \Delta$ is a free abelian group. If $\Delta$ was torsion-free, we could directly isotope $\rho$ to the trivial map. Indeed, for any $k\geq 0$, we have
	\begin{equation} \text{Hom}(\bbZ^k, T)\cong T^k.\label{eq:component} \end{equation}
Therefore the following claim will suffice to finish the proof.
\begin{claimnr} There exists a finite cover $M''\rightarrow M'$ of $M'$ such that
	\begin{enumerate}[(i)]
		\item $p'$ lifts to a strongly regular self-cover $p'':M''\rightarrow M''$,
		\item The action of $T$ on $M'$ lifts to an action of $T$ on $M''$, and
		\item The monodromy of $M''\rightarrow M''\slash T$ is given by $\rho|_{\pi_1(M'')}$ and factors through a torsion-free group.
	\end{enumerate}
\end{claimnr}
\begin{proof} We need to investigate the splitting $\pi_1(M')=\pi_1(T)\times \pi_1(N)$ more carefully. This splitting is not canonical, but we claim that nevertheless we can choose a splitting that is $\varphi$-invariant. Indeed, we identify $\pi_1(T)$ with a subgroup of $\pi_1(M')$ by using an orbit map
	$$\theta_{x_0}: T\rightarrow M', \quad g\mapsto gx_0$$
for some fixed $x_0\in M'$. Any two orbit maps induce the same map on fundamental groups, so the image of $\theta_{x_0}$ does not depend on $x_0$. The image $\theta_{x_0\ast} \pi_1(T)$ is then $\varphi$-invariant because $p'$ maps orbits to orbits.

Further we can identify $\pi_1(N)$ with the stable image $\varphi^\infty(\pi_1(M'))$ of $\pi_1(M')$ because the composition
	$$\varphi^\infty(\pi_1(M'))\hookrightarrow \pi_1(M')\rightarrow \pi_1(N)$$
is an isomorphism. The inverse $\pi_1(N)\rightarrow \varphi^\infty(\pi_1(M'))$ is therefore a section of $\pi_1(M')\rightarrow \pi_1(N)$. So
	$$\pi_1(M')\cong \theta_{x_0\ast} \pi_1(T)\times \varphi^\infty(\pi_1(M'))$$
is the desired $\varphi$-invariant splitting. Our next goal is to define the cover $M''\rightarrow M'$.

Identifying $\pi_1(N)$ with $\varphi^\infty(\pi_1(M'))$, we have that $\varphi$ restricts to an automorphism of $\pi_1(N)$. Therefore $\varphi$ descends to some automorphism $\varphi_\Delta$ of the abelianization $\Delta$ of $\pi_1(N)$. As above we write
	$$\Delta=\text{Tor}(\Delta)\oplus \Delta_\infty.$$
Note that $\varphi_\Delta$ preserves $\text{Tor}(\Delta)$, but does not necessarily preserve $\Delta_\infty$. Write $C:=[\Delta:\Delta_\infty]$, and let $\Delta_C$ be the intersection of all subgroups of $\Delta$ of index $C$. Then $\Delta_C$ is torsion-free (because it is contained in $\Delta_\infty$). Further $\Delta_C$ is a finite index characteristic subgroup of $\Delta$, and hence  $\varphi_\Delta$ restricts to an automorphism of $\Delta_C$. 

Let $\Gamma''$ be the preimage in $\pi_1(M')$ of $\Delta_C$ under the composition
	$$\pi_1(M')\rightarrow \pi_1(N)\rightarrow \Delta.$$
Then $\Gamma''$ is a finite index subgroup of $\pi_1(M')$ and hence corresponds to a cover $M''\rightarrow M'$. The self-cover $p':M'\rightarrow M'$ lifts to self-cover $p'':M''\rightarrow M''$ because $\varphi(\Gamma'')\subseteq \Gamma''$. This shows that $M''\rightarrow M'$ satisfies Property (i) of the claim.

Our next task is to show that $M''\rightarrow M'$ satisfies (ii), i.e. the action of $T$ on $M'$ lifts to an action of $T$ on $M''$. Consider the group of lifts $H$ of $T$ to $M''$. Then $H$ is an extension given by
	$$1\rightarrow \pi_1(M')\slash\pi_1(M'') \rightarrow H\rightarrow T\rightarrow 1.$$
Let $H^0$ be the component of $H$ containing the identity. Then the action of $T$ on $M'$ lifts to an action of $T$ on $M''$ precisely when the map $H\rightarrow T$ restricts to an isomorphism $H^0\rightarrow T$. In turn $H^0\rightarrow T$ is an isomorphism precisely when the induced map $\pi_1(H^0)\rightarrow \pi_1(T)$ is an isomorphism. 

Note that $\pi_1(H^0)\rightarrow \pi(T)$ is automatically injective (because $H^0\rightarrow T$ is a cover). To show it is also surjective, consider an orbit map $\theta: H^0\rightarrow M''$. Then the map on fundamental groups induced by $\theta$ is a map
	$$\pi_1(H^0)\rightarrow \Gamma''\cong \pi_1(T)\times \pi_1(N'')$$
such that the square
	$$\xymatrix{
	\pi_1(H^0) \ar[r] \ar[d] & \Gamma'' \cong \pi_1(T)\times \pi_1(N'') \ar[d]\\
	\pi_1(T) \ar[r] & \Gamma\cong \pi_1(T)\times \pi_1(N)
	}$$
commutes. This shows that $\pi_1(H^0)$ surjects onto the first factor of $\pi_1(T)\times\pi_1(N)$, so that $\pi_1(H^0)\rightarrow \pi_1(T)$ is surjective. We already had injectivity, so $\pi_1(H^0)\rightarrow \pi_1(T)$ is an isomorphism, as desired.

Finally we show that $M''\rightarrow M''\slash T$ has trivial monodromy. Indeed, let $\rho''$ be the monodromy of $M''\rightarrow M''\slash T$. Then $\rho''$ is just given by the composition
	$$\Gamma'' \hookrightarrow \Gamma' \overset{\rho}{\to} T.$$
Obviously $\rho$ factors through the abelianization $\Gamma'\rightarrow \Delta$ and descends to a map
	$$\rho_\Delta:\Delta\cong \Delta_\infty\oplus \text{Tor}(\Delta)\to T,$$
and we have already seen above (in Equation \eqref{eq:component}) that $\Delta_\infty\subseteq \ker(\rho_\Delta)$. On the other hand, we defined $M''$ so that the composition of the natural inclusion and the abelianization map
	$$\pi_1(M''\slash T)\hookrightarrow \pi_1(M'\slash T)\overset{\text{ab}}{\to} \Delta$$
has image $\Delta_C$. Since $\Delta_C\subseteq \Delta_\infty$,  it follows that $\rho''$ is trivial, as desired.  This completes the proof of Theorem \ref{thm:kaehler}. \end{proof}

\section{Residual self-covers} \label{sec:respf}

In this section we prove Corollaries \ref{cor:charres} (about residual self-covers) and \ref{cor:charexp} (that relates strongly regular expanding maps and linear expanding maps on tori).

\begin{proof}[Proof of Corollary \ref{cor:charres}] Let $X$ be a finite CW-complex and $p:X\rightarrow X$ a strongly regular, residual self-cover. As usual we write $\Gamma=\pi_1(X)$ and $\varphi:=p_\ast$ for the map induced by $p$ on fundamental groups. We need to show that $\Gamma$ is abelian.

By Theorem \ref{thm:group}, there exists a surjection
	$$q:\pi_1(X)\rightarrow \bbZ^n$$
(for some $n\geq 1$) such that $\varphi$ restricts to an automorphism on $\ker(q)$. Therefore we have $\ker(q)\subseteq \varphi^m(\Gamma)$ for every $m\geq 1$. Since $p$ is residual, we have $\cap_m \varphi^m(\Gamma)=1$, so $\ker(q)$ must be trivial. This shows that $q$ is an isomorphism, as desired. 

Now suppose that $X$ is a closed aspherical manifold. Then $X$ is a closed aspherical manifold with abelian fundamental group. The Borel conjecture for manifolds with abelian fundamental groups (proven for $n>4$ by Farrell-Hsiang \cite{borelab}, for $n=4$ uses Freedman's work, and for $n=3$ uses Perelman's solution of Thurston's geometrization conjecture) exactly states that such a manifold is homeomorphic to a torus. \end{proof}

\begin{proof}[Proof of Corollary \ref{cor:charexp}] An argument of Shub shows that any manifold admitting an expanding map is aspherical \cite{shubexp}: Namely, $p$ lifts to a homeomorphism $\widetilde{p}$ of $\widetilde{M}$. Now fix any $x\in\widetilde{M}$ and $0<r<\textrm{injrad}(M)$. Then $B(x,r)\subseteq\widetilde{M}$ is an embedded ball and 
	$$\widetilde{M}=\bigcup_{n\geq 0} \tilde{p}^n(B(x,r))$$
is a union of embedded balls and hence contractible.
	
	Further it is easy to see that $p$ is a self-cover because it is a local diffeomorphism of the closed manifold $M$. Since $p$ is expanding, it has to be a residual self-cover: Indeed, $\tilde{p}$ is $p_\ast:\pi_1(M)\rightarrow \pi_1(M)$ equivariant, and hence
		$$p_\ast^n(\pi_1(M))=\{\gamma\in\pi_1(M)\mid \gamma x \in \tilde{p}^n(\pi_1(M)\cdot x)\}.$$
	Since $\widetilde{p}$ is globally expanding and the orbit $\pi_1(M)\cdot x$ is a discrete subset of $\widetilde{M}$, it follows that for every $\gamma\in\pi_1(M)$ there exists $n\geq 1$ such that $\gamma x\notin \tilde{p}^n(\pi_1(M)\cdot x)$. 
	
	We characterized nontrivial, strongly regular, residual self-covers in Corollary \ref{cor:charres}. Using the part of this result that applies to aspherical manifolds, we find that $M$ is finitely (topologically) covered by a torus $T$ and $p$ lifts to an expanding map $p':T\rightarrow T$ (with respect to a potentially exotic smooth structure). This implies that $p'$ is homotopic to a linear expanding map on $T$. Shub was able to show that $p'$ is not just homotopic, but also topologically conjugate to a linear expanding map on $T$ \cite{shubexp}.
	\end{proof}

\bibliographystyle{alpha}
\bibliography{strongregcover}

\end{document}